\documentclass[twoside]{article}
\usepackage{amsmath,amsfonts,amssymb,amsthm}
\usepackage{hyperref}
\usepackage[a4paper]{geometry}
\usepackage[color,matrix,arrow,all]{xy}
\usepackage{stackrel,relsize}
\usepackage{tikz}
\usepackage{mathdots}
\newcommand{\disp}{\displaystyle}
\newcommand{\ds}{\displaystyle\sum}

\newcommand{\nat}{\mathbb{N}}
\newcommand{\integ}{\mathbb{Z}}
\newcommand{\orb}{\mathcal{O}}
\newcommand{\op}{\operatorname}
\newcommand{\und}{\underline}
\newcommand{\ove}{\overline}
\DeclareMathOperator{\Dim}{\mathbf{dim}}

\newcommand{\Hom}{\operatorname{Hom}}

\newcommand{\Ext}{\operatorname{Ext}}
\newcommand{\GL}{\operatorname{GL}}

\newcommand{\Rep}{\operatorname{Rep}}

\newcommand{\Sym}{\operatorname{Sym}}
\newcommand{\rank}{\operatorname{rank}}
\newcommand{\A}{\mathbb{A}}

\newtheorem{theorem}{Theorem}[section]

\newtheorem{lemma}[theorem]{Lemma}
\newtheorem{proposition}[theorem]{Proposition}
\newtheorem{corollary}[theorem]{Corollary}

\theoremstyle{definition}
\newtheorem{remark}[theorem]{Remark}
\newtheorem{definition}[theorem]{Definition}

\newtheorem{example}[theorem]{Example}

\def\presuper#1#2%
  {\mathop{}%
   \mathopen{\vphantom{#2}}^{#1}%
   \kern-\scriptspace%
   #2}

\begin{document}

\title{Free resolutions of orbit closures of Dynkin quivers}
\author{Andr\'as C L\H{o}rincz, Jerzy Weyman}
\date{}

\maketitle

\begin{abstract}
We use the Kempf-Lascoux-Weyman geometric technique in order to determine the minimal free resolutions of some orbit closures of quivers. As a consequence, we obtain that for Dynkin quivers orbit closures of 1-step representations are normal with rational singularities. For Dynkin quivers of type $\A$, we describe explicit minimal generators of the defining ideals of orbit closures of 1-step representations. Using this, we provide an algorithm  for type $\A$ quivers for describing an efficient set of generators of the defining ideal of the orbit closure of any representation.
\end{abstract}


\vspace*{5mm}

\section*{Introduction}
\label{sec:intro}

The geometric properties of orbit closures of quivers (also known as quiver loci) have been studied extensively, and it is an active area of research (see \cite{expo} for an exposition).  Several results are known in the Dynkin case. It has been shown (see \cite{abea,orb1,orb2,ryan,lak}) that for quivers of type $\A$ and $\mathbb{D}$ orbit closures have rational singularities (in particular, are normal and Cohen-Macaulay). Furthermore, for equioriented type $\A$ quivers it was shown in \cite{lak} that singularities of orbit closures are identical to singularities of Schubert varieties. For Dynkin quivers of type $\mathbb{E}$, it is still an open problem whether orbit closures are normal, Cohen-Macaulay or have rational singularities, and only partial results are known (\cite{lol,kavita, codim}).

The Kempf-Lascoux-Weyman geometric technique is a generalization of Lascoux's calculation of the resolutions of determinantal varieties \cite{lasc}.  In this paper we investigate the minimal free resolutions of the defining ideals of orbit closures that are $1$-\textit{step} using this technique. This was investigated before for source-sink quivers by K. Sutar \cite{kavita2,kavita}. We generalize the results in \cite{kavita} and show, in particular, that 1-step orbit closures are normal with rational singularities for all Dynkin quivers (Theorem \ref{thm:dynkin}). Besides yielding minimal free resolutions, this approach of studying the geometry of orbit closures has also the advantage that it is uniform with respect to all Dynkin types (and extend Dynkin types).

The article is organized as follows. In Section \ref{sec:quiv} we give a short background on quivers from a geometric point of view. Then we define 1-step representations (Definition \ref{def:step}) and give some criteria for representations to be 1-step (Propositions \ref{prop:step1},\ref{prop:step2},\ref{prop:step3}).

In Section \ref{sec:geom} we consider the geometric technique in our context. The technique provides a complex (\ref{eq:basic}) whose terms are obtained from cohomologies on (products of) Grassmannians. In good situations these complexes give minimal free resolutions of orbit closures. The terms of the complex can be computed using Bott's Theorem \ref{thm:bott}, which involves combinatorics with partitions.

In Section \ref{sec:main} we provide the main results on minimal free resolutions. A sequence of lemmas is followed by the main technical result (Theorem \ref{thm:main}), which gives a direct relation between the combinatorics behind Bott's Theorem and the Euler form of $Q$. Using this, we readily prove that 1-step representations have rational singularities (in particular, are normal and Cohen-Macaulay) when $Q$ is a Dynkin quiver (Theorem \ref{thm:dynkin}). Moreover, we get minimal free resolutions whose terms can be computed using Bott's theorem. Also, we give the analogous results for the normalizations of 1-step orbit closures in the extended Dynkin case (Theorem \ref{thm:extended}).

In Section \ref{sec:type} we give more explicit results for type $\A$ Dynkin quivers. Using the first term $F_1$ in the minimal free resolution, we identify the minimal defining equations of 1-step orbit closures as certain minors that come from rank conditions (Theorems \ref{thm:min},\ref{thm:rank}). Using these results, prove that the orbit closure of any representation for a type $\A$ quiver can be written as a scheme-theoretic intersection of 1-step orbit closures (Theorem \ref{thm:scheme}). This gives an algorithm for a type $\A$ quiver for finding an efficient set of generators for the defining ideal of the orbit closure of \textit{any} representation.

\section{Quivers and 1-step representations}
\label{sec:quiv}

Throughout we work over an algebraically closed field $k$ of characteristic $0$. A quiver $Q$ is an oriented graph, i.e. a pair $Q=(Q_0,Q_1)$ formed by a finite set of vertices $Q_0$ and a finite set of arrows $Q_1$. An arrow $a$ has a head $ha$, and tail $ta$, that are elements in $Q_0$:

\[\xymatrix{
ta \ar[r]^{a} & ha
}\]

A representation $V$ of $Q$ is a family of finite dimensional vector spaces $\{V_x\,|\, x\in Q_0\}$ together with linear maps $\{V(a) : V_{ta}\to V_{ha}\, | \, a\in Q_1\}$. The dimension vector $\Dim V\in \nat^{Q_0}$ of a representation $V$ is the tuple $\Dim V=(\dim V_x)_{x\in Q_0}$. A morphism $\phi:V\to W$ of two representations $V,W$ is a collection of linear maps $\phi = \{\phi(x) : V_x \to W_x\,| \,x\in Q_0\}$, with the property that for each $a\in Q_1$ we have $\phi(ha)V(a)=W(a)\phi(ta)$. Denote by $\Hom_Q(V,W)$ the vector space of morphisms of representations from $V$ to $W$. For two vectors $\alpha, \beta\in \integ^{Q_0}$, we define the Euler product $\langle \alpha, \beta \rangle = \ds_{x\in Q_0} \alpha_x \beta_x - \ds_{a\in Q_1} \alpha_{ta} \beta_{ha}$. The Euler form $E_Q$ of a quiver $Q$ is the quadratic map $\integ^{Q_0}\to \integ$ given by
$E_Q(\alpha)= \langle \alpha, \alpha \rangle$.

It is known that $E_Q$ is positive definite (resp. positive semi-definite)  iff $Q$ is a Dynkin (resp. extended Dynkin) quiver. For more on quivers cf. \cite{elements,quad}.

We form the affine space of representations with dimension vector $\alpha\in \nat^{Q_0}$ by
$$\Rep(Q,\alpha):=\displaystyle\bigoplus_{a\in Q_1} \Hom(k^{\alpha_{ta}},k^{\alpha_{ha}}).$$
The group 
$$\GL(\alpha):= \prod_{x\in Q_0} \GL(\alpha_x)$$
acts by conjugation on $\Rep(Q,\alpha)$ in the obvious way. Under the action $\GL(\alpha)$ two elements lie in the same orbit iff they are isomorphic as representations. 

For any two representations $V$ and $W$, we have the following exact sequence:
\begin{equation}\label{eq:ringel}
\begin{array}{rlc}
0 \to \Hom_Q (V,W) \longrightarrow \displaystyle\bigoplus_{x \in Q_0}& \!\!\!\!\!\Hom(V(x),W(x)) & \\
& \stackrel{d^V_W}{\longrightarrow}  \displaystyle\bigoplus_{a\in Q_1} \Hom(V(ta),W(ha)) \longrightarrow \Ext_Q(V,W)\to 0, &
\end{array}
\end{equation}
where $d_W^V$ is given by
$$\{\phi(x)\}_{x\in Q_0} \mapsto \{\phi(ha)V(a) - W(a)\phi(ta)\}_{a\in Q_1}.$$
The exact sequence (\ref{eq:ringel}) gives $\langle \Dim V,\Dim W \rangle = \dim \Hom_Q(V,W) - \dim \Ext_Q(V,W)$.

Taking $V=W$ the map $d^V_V$ is the differential at the identity of the orbit map
$$g\mapsto g\cdot V \in \Rep(Q,\Dim V),$$
and we have a natural identification of the normal space
$$\Ext_Q(V,V)\cong \Rep(Q,\Dim V)/T_V(\mathcal{O}_V),$$
where $\mathcal{O}_V$ is the orbit of $V$. In particular, the codimension of the orbit closure $\overline{\orb}_V$ in $\Rep(Q,\alpha)$ is equal to $\dim \Ext_Q(V,V)$.

Take two dimension vectors $\beta,\gamma$ with $\alpha=\beta+\gamma$. In case of Dynkin quivers, M. Reineke \cite{reineke} constructs desingularizations of for all orbit closures. These are total spaces of some vector bundles over a product of flag varieties. We consider the simplest non-trivial case when these flag varieties are all Grassmannians, but relaxing the condition of desingularization. Then the vector bundles are the incidence varieties $Z(Q,\beta \subset \alpha)$ that were introduced in \cite{scho}. Namely, let $\op{Gr}(\beta,\alpha)$ denote the product of Grassmannians $\prod_{x\in Q_0} \op{Gr}(\beta_x,\alpha_x)$. Then the space $Z(Q,\beta \subset \alpha)$ is the vector subbundle of
$$\Rep(Q,\alpha)\times \op{Gr}(\beta,\alpha)$$
consisting of points $(V,\{R_x\})$ such that the collection of subspaces $\{R_x\}_{x\in Q_0}$ forms a subrepresentation of $V$. 

If $Q$ is a Dynkin quiver then $\Rep(Q,\alpha)$ has finitely many orbits. Hence the image under the projection
$$q: Z(Q,\beta \subset \alpha) \to \Rep(Q,\alpha)$$
is an orbit closure for any $\beta$.

\begin{definition}\label{def:step}
Let $Q$ be any quiver. We say a representation $V\in \Rep(Q,\alpha)$ (or the orbit closure $\overline{\orb}_V$) is 1-\textit{step} if  $\overline{\orb}_V$ is the image of the projection
$$q: Z(Q,\beta \subset \alpha) \to \Rep(Q,\alpha)$$
for some dimension vector $\beta$.
\end{definition}

We note that we do not require the projection above to be birational, hence generalizing the definition in \cite{kavita}. 

For Dynkin quivers any dimension vector $\beta$ gives a 1-step representation. 

For the $\A_2$ and non-equioriented $\A_3$ quivers all representations are 1-step (see \cite{kavita2}). However, this fails for other quivers. For example,  if we take the equioriented $\A_3$ quiver, then the sum of simples $V=S_1\oplus S_2 \oplus S_3$ is not 1-step.

We say that a representation $T$ is \textit{generic} if $\Ext(M,M)=0$, or equivalently, if the orbit $\orb_{M}$ is open in $\Rep(Q,\Dim M)$.

Now we recall generic extensions \cite[Definition 2.2]{generic}. Let $M,N$ be representations of a Dynkin quiver $Q$. We say $V$ is the \textit{generic extension} of $M$ by $N$ if we have an exact sequence
$$0 \to N \to V \to M \to 0,$$
and for any other exact sequence $0 \to N \to W \to M \to 0$ we have $W\in \overline{\orb}_V$. In this case we write $V:=M*N$.

The following theorem follows essentially from \cite[Proposition 2.4]{generic} and gives a more explicit description of 1-step representations:

\begin{proposition}\label{prop:step1}
Let $Q$ be Dynkin quiver and $V$ be the 1-step representation of $Q$ with projection
$$q: Z(Q,\beta \subset \alpha) \twoheadrightarrow \overline{\orb}_V.$$
Consider the generic representations $T_{\beta},T_{\gamma}$ with dimension vectors $\beta,\gamma$, respectively, where $\alpha=\beta+\gamma$. Then $V=T_\gamma*T_\beta$.
\end{proposition}

In other words, a representation $V$ of a Dynkin quiver is 1-step if and only if it can be decomposed as $V=M*N$, where $M$ and $N$ are generic.

We note that the paper \cite{generic} provides a 'straightening' algorithm for computing generic extensions, as well as there are several algorithms computing generic representations (e.g. see \cite{canon} -- for a simple algorithm for type $\A$ Dynkin quivers, see \cite{abeasis}). 

Next, we give some more concrete examples of 1-step representations:

\begin{proposition}\label{prop:step2}
Let $Q$ be any quiver, and $V$ a representation of $Q$ with a decomposition $V=M\oplus N$, where $M,N$ are generic and $\Ext(M,N)=0$. Then $V$ is a 1-step representation with $\beta=\Dim N$ and the fiber $q^{-1}(V)$ is irreducible.
\end{proposition}

\begin{proof}
We consider the proper map
$$q: Z(Q,\beta \subset \alpha) \to \Rep(Q,\alpha).$$
with  $\beta=\Dim N$. Clearly, $V\in \op{Im} q$, hence $\overline{\orb}_V \subset \op{Im} q$. By the proof of \cite[Theorem 3.3]{scho}, the codimension of $\op {Im} q$ in  $\Rep(Q,\alpha)$ equals $\dim \Ext(N,M)$. But the codimension of  $\overline{\orb}_V$ in  $\Rep(Q,\alpha)$ is $\dim \Ext(V,V)=\dim \Ext(N,M)$, hence  $ \op{Im} q=\overline{\orb}_V$ proving that $V$ is 1-step.

To prove the second claim, we first show that if  $N'$ is a subrepresentation of $V$ with $N'\in \Rep(Q,\beta)$, then we must have $\dim\Hom_Q(N',V)=\dim\Hom_Q(N,V)$.

We can write 
$$Z(Q,\beta \subset \alpha)= G\times_P (\Rep(Q,\beta) \times \Rep(Q,\gamma) \times W),$$
where $W$ is the affine space $\prod_{a\in Q_1}\Hom_k(k^{\gamma_{ta}}, k^{\beta_{ha}})$ and $P$ is the appropriate parabolic subgroup of $\GL(\alpha)$ (see \cite{scho}). The assumption on $N'$ implies that the image of $q$ restricted to the closed subvariety 
$$ G\times_P (\overline{\orb}_{N'} \times \Rep(Q,\gamma) \times W)$$ 
is $\overline{\orb}_V$. The preimage of the open $\orb_V$ must intersect the open  $G\times_P (\orb_{N'} \times \orb_M \times W)$, hence we have an exact sequence
$$0\to N' \to V \to M \to 0.$$
Applying to the sequence the functor $\Hom(-,V)$ and using that $\Ext(M,V)=0$, we obtain $\dim\Hom_Q(N',V)=\dim\Hom_Q(N,V)$.

Since $N$ is generic, semi-continuity implies that the set 
$$U:=\{X\in \Rep(Q,\beta),\ | \,\dim\Hom_Q(X,V)=\dim\Hom_Q(N,V)\}$$ 
is an open subset of $\Rep(Q,\beta)$, hence an irreducible variety.

Take the space of maps $\Hom(\beta,\alpha):= \bigoplus_{x\in Q_0} \Hom(k^{\beta_x},k^{\alpha_x})$, and let $Z$ be the closed subset of the $\Rep(Q,\beta)\times\Hom(\beta,\alpha)$ of elements $(X,f)$ such that $f\in \Hom_Q(X,V)$. The subset $Z^0$ of $Z$ of pairs $(X,f)$ with $f$ injective is open and non-empty. By \cite{quivgrass}, the fiber $q^{-1}(V)$ can be realized as the geometric quotient of $Z^0$ by $\GL(\beta)$. Hence it is enough to show that $Z^0$ is irreducible.

By the above, the preimage $p^{-1}(U)$ of the projection $p:Z \to \Rep(Q,\beta)$ satisfies $Z^0\subset  p^{-1}(U)$, hence it is enough to see that $p^{-1}(U)$ is irreducible.

It is easy to see that $p: p^{-1}(U)\to U$ is in fact a vector bundle (see \cite[Lemma 2.1]{bong}). Since $U$ is irreducible, this implies that $p^{-1}(U)$ is irreducible as well.
\end{proof}

\vspace{0.05in}

We note the similarity of the above result regarding the fiber to \cite[Proposition 3.1]{quivgrass}.

\vspace{0.1in}

Not all 1-step representations are of the above form. For example, let $Q$ be the equioriented $\A_4$ quiver 
$$1\to 2 \to 3 \to 4,$$
with $V$ the 1-step representation obtained from $\beta = (1,0,1,0)$ and $\alpha=(1,1,1,1)$. Then $V$ decomposes as $V=1000\oplus 0110 \oplus 0001$, but it cannot be written as a decomposition $V=M\oplus N$ with $M,N$ generic and $\Ext(M,N)=0$.

\vspace{0.15in}

Now let $Q$ be an arbitrary quiver, $T$ a generic representation with full support $Q$ and $\alpha=\Dim T$. Writing $T=T_1^{\lambda_1}\oplus\dots\oplus T_r^{\lambda_r}$ with $\lambda_i>0$ and $T_1,\dots, T_r$ pair-wise non-isomorphic indecomposables, we must have $\Ext(T_i,T_j)=0$ for all $i,j$. We investigate the irreducible components of the complement of $\orb_T$ in $\Rep(Q,\alpha)$. Whenever such a component is an orbit closure $\overline{\orb}_V$ of a representation $V$, we call $V$ a \textit{subgeneric} representation.

In \cite{open} it is shown that if $T$ is so-called \textit{stable}, then all the irreducible components of the complement of $\orb_T$ are orbit closures. Roughly speaking, $T$ is stable when the multiplicities $\lambda_i$ are large enough (for a precise definition of stability, see \cite{open}). In case of Dynkin quivers, any generic representation is stable.

Sub-generic representations of codimension $1$ are given by semi-invariants, and their geometry has been studied using Bernstein-Sato polynomials \cite{lol,bfun}.

\begin{proposition}\label{prop:step3}
Let $T$ be a stable generic representation. Then subgeneric representations in $\Rep(Q,\Dim T)$ are 1-step.
\end{proposition}

\begin{proof}
Assume first that $V$ is a subgeneric representation of codimension 1. By \cite{open}, $V$ can be written in the form $V=T'\oplus Z \oplus R$, with $T',Z,R$ generic and $\Ext(Z,T')=k$. Hence we can apply Proposition \ref{prop:step2} with $M=T'\oplus R$ and $N=Z$.

Now if $V$ is subgeneric of codimension larger than 1, then by \cite{open} it can be written as either $V=T_i^{\lambda_i+1}\oplus X \oplus R$, where the only non-trivial extension between $T_i,X,R$ is $\Ext(T_i,X)=k$, or $V=T_i^{\lambda_i+1}\oplus Y \oplus R$ with the only non-trivial extension group $\Ext(Y,T_i)=k$. We can apply Proposition \ref{prop:step2} in the former case with $M=X$ and $N=T_i^{\lambda_i+1}\oplus R$, while in the latter with $M=T_i^{\lambda_i+1}\oplus R$ and $N=Y$.
\end{proof}

\section{The geometric technique}
\label{sec:geom}

\vspace{0.05in}

Throughout we incorporate much of the notation used in \cite{kavita}. For more on the geometric technique, see \cite{jerzy}. 

Let $V\in \Rep(Q,\alpha)$ be a 1-step representation with
$$q: Z(Q,\beta \subset \alpha) \twoheadrightarrow  \overline{\orb}_V$$
for some dimension vector $\beta$ and let $\alpha=\beta+\gamma$.

We can identify the affine space $\Rep(Q,\alpha)$ with $\bigoplus_{a\in Q_1} V^*_{ta}\otimes V_{ha}$. Denote by $A$ the coordinate ring of $\Rep(Q,\alpha)$.

\vspace{0.05in}

For a vertex $x\in Q_0$, denote by $\mathcal{R}_x$ (resp. $\mathcal{Q}_x$) the tautological bundle (resp. factorbundle) on $\op{Gr}(\beta_x,\alpha_x)$. We view $\Rep(Q,\alpha)\times \op{Gr}(\beta,\alpha)$ as the total space of the trivial bundle $\mathcal{E}$ and $Z(Q,\beta \subset \alpha)$ as the total space of some subbundle $\mathcal{S}$ of $\mathcal{E}$. Let $\xi$ denote the dual of the factorbundle $\mathcal{E}/\mathcal{S}$. More explicitly, it is given by
$$\xi = \bigoplus_{a\in Q_1} \mathcal{R}_{ta}\otimes \mathcal{Q}_{ha}^*.$$

Applying \cite[Theorem 5.1.2]{jerzy}, we consider the complex $F_\bullet$ with terms
\begin{equation}\label{eq:basic}
F_i = \bigoplus_{j\geq 0} H^j (\operatorname{\op{Gr}(\beta,\alpha)}, \bigwedge^{i+j} \xi)\otimes A(-i-j).
\end{equation}

We use the following version of \cite[Theorem 5.1.3]{jerzy} without assuming that $q$ is birational.

\begin{theorem}\label{thm:rational} Using the notation above, we have the following:
\begin{itemize}
\item[(a)] If $F_i=0$ for all $i<0$, and the fiber $q^{-1}(V)$ is connected, then $F_\bullet$ is a finite minimal free resolution of the normalization of $\overline{\orb}_V$, and the normalization has rational singularities.
\item[(b)] If $F_i=0$ for all $i<0$ and $F_0=A$, then $\overline{\orb}_V$ is normal and it has rational singularities.
\end{itemize}
\end{theorem}

\begin{proof}
\begin{itemize}
\item[(a)] Put $Z:= Z(Q,\beta\subset \alpha)$.  We factor $q$ through the normalization $Y$ of $\overline{\orb}$, then apply Stein factorization to obtain $Z\xrightarrow{f} Y'\to Y$, where $Y'$ is an integral scheme with $f_* \mathcal{O}_Z = \mathcal{O}_{Y'}$. The map $Y'\to Y$ is birational since the general fiber of $q$ is connected. This shows that $Y'\simeq Y$. The statement about minimal free resolution follows as in \cite[Theorem 5.1.3]{jerzy}. and the statement about rational singularities follows from \cite{kovacs}.
\item[(b)] Let $\overline{\orb}_V= \op{Spec} A/I$, where $I$ is a prime ideal. Applying Stein factorization we have $Z\to X \to \op{Spec} A/I$, where $X$ is an integral affine scheme. Since $F_0=A$ and  $F_i=0$ for all $i<0$,  \cite[Theorem 5.1.3]{jerzy} implies that $X=\op{Spec} A/J$ for some prime ideal $J$. But $\op{Spec} A/J \to \op{Spec} A/I$ is surjective, so it must be an isomorphism. The rest follows as in part (a).
\end{itemize}
\end{proof}

A partition (with $n$ parts) $\lambda = (\lambda_1,\lambda_2,\dots,\lambda_n)$ is a non-increasing sequence of non-negative integers. For a partition $\lambda$ we associate its corresponding Young diagram that consists of $\lambda_i$ boxes in the $i$th row. We denote the number of boxes by $|\lambda|:= \lambda_1+\dots+\lambda_n$. We denote by $u_\lambda$ the size of the Durfee square of $\lambda$, that is, the biggest square fitting inside of the Young diagram of $\lambda$. Its defining property is $\lambda_{u_\lambda}\geq u_\lambda$ and $\lambda_{u_\lambda+1}\leq u_\lambda$.

Let $\lambda^+$ be the partition $(\lambda_1-u_{\lambda},\lambda_2-u_{\lambda},\dots,\lambda_{u_\lambda}-u_{\lambda})$ and $\lambda^{-}$ the partition $(\lambda_{u_\lambda+1},\dots,\lambda_n)$. Hence we can view the Young diagram of $\lambda$ as the composite of the $3$ parts $\lambda^+,\lambda^{-}$ and a $u_\lambda\times u_\lambda$ square:
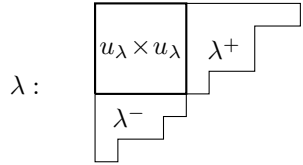
\[ \hspace{-0.3in} \lambda: \hspace{0.2in}
\begin{aligned}
\begin{tikzpicture}
\draw (0,0.9)--(0,0)--(0.3,0)--(0.3,0.3)--(0.9,0.3)--(0.9,0.6)--(1.2,0.6)--(1.2,0.9);
\node at (0.45,0.6) {$\lambda^-$};
\draw [thick] (0,0.9)--(0,2.1)--(1.2,2.1)--(1.2,0.9)--(0,0.9);
\node at (0.6,1.5) {$u_\lambda \! \times \! u_\lambda$};
\draw(1.2,0.9)--(1.5,0.9)--(1.5,1.2)--(2.1,1.2)--(2.1,1.8)--(2.7,1.8)--(2.7,2.1)--(1.2,2.1);
\node at (1.7,1.5) {$\lambda^+$};
\end{tikzpicture}
\end{aligned}
\]

The conjugate partition $\lambda'=(\lambda_1',\dots, \lambda_m')$ is a partition with $\lambda'_i$ being the number of boxes in the $i$th column of the Young diagram of $\lambda$. Note that $u_{\lambda}=u_{\lambda'}$.

Denote by $-\lambda$ the non-increasing sequence of non-positive integers $(-\lambda_n,-\lambda_{n-1},\dots, -\lambda_1)$.

A weight (with $n$ parts) is any sequence of integers  $\delta:=(\delta_1, \delta_2,\dots ,\delta_n)$. For a weight $\delta$ and a number $u\in\integ$ we denote by $\delta+u$ the weight $(\delta_1+u,\delta_2+u,\dots, \delta_n+u)$. 

We consider the action of the symmetric group $\Sigma_n$ on weights defined as follows: a transposition $\sigma_i=(i,i+1)$ acts according to the exchange rule
$$\sigma_i \cdot \delta = (\delta_1,\dots ,\delta_{i-1},\delta_{i+1}-1,\delta_{i}+1,\delta_{i+2},\dots,\delta_n).$$ 
Let $N(\delta)$ be length of the (unique) permutation $\sigma\in \Sigma_n$ such that the sequence $\sigma \cdot \delta$ is non-increasing, if there exists such a permutation, otherwise put $N(\delta):=-\infty$. In other words, $N(\delta)$ is the minimal number of exchanges applied to $\delta$ that turn it non-increasing. Clearly, we have $N(\delta)=N(\delta+u)$, for any $u\in\integ$.

\vspace{0.1in}

For any partition $\lambda$, we denote by $S_\lambda$ the corresponding Schur functor. Now for $1\leq t \leq \dim \xi$, we decompose $\bigwedge^{t} \xi$ using Cauchy's formula (see \cite{jerzy}) as in \cite{kavita}:
$$\bigwedge^{t} \xi = \bigoplus_{\sum_{a\in Q_1}\!\! k_a = t}\, \bigotimes_{a\in Q_1} \bigwedge^{k_a} \mathcal{R}_{ta}\otimes \mathcal{Q}_{ha}^*=$$
$$= \bigoplus_{\sum_{a\in Q_1}\!\! |\lambda_a| = t}\, \bigotimes_{a\in Q_1} S_{\lambda(a)} \mathcal{R}_{ta}\otimes S_{\lambda(a)'} \mathcal{Q}_{ha}^*=$$
\begin{equation}\label{eq:cauchy}
=\bigoplus_{\sum_{a\in Q_1} \!\! |\lambda_a| = t}\, \bigotimes_{x\in Q_0}\left[ \bigotimes_{a\in Q_1 : ta=x} S_{\lambda(a)}\mathcal{R}_{x}\otimes \bigotimes_{b\in Q_1 : hb=x} S_{\lambda(b)'} \mathcal{Q}_{x}^*\right].
\end{equation}

For any arrow $a\in Q_1$, we choose a partition $\lambda(a)$ with its Young diagram having at most $\beta_{ta}$ rows and at most $\gamma_{ha}$ columns. Now to any vertex $x\in Q_0$, we associate two partitions $\mu(x)$ and $\nu(x)$ as follows. Let $a_1,\dots , a_k$ be all outgoing arrows and $b_1,\dots, b_l$ all the incoming arrows at $x$. Then choose $\mu(x)$ (resp. $\nu(x)$) a partition corresponding to any Young diagram occurring in the Littlewood-Richardson product of $\lambda(a_1),\dots , \lambda(a_k)$ (resp. $\lambda(b_1)',\dots,\lambda(b_l)'$) with at most $\beta_x$ rows (resp. $\gamma_x$ rows). Denote the collection of all these partitions by
\begin{equation}
\label{eq:coll}
\underline{\lambda}:= ( (\lambda(a))_{a\in Q_1},(\mu(x))_{x\in Q_0},(\nu(x))_{x\in Q_0}).
\end{equation}

To compute the cohomology of a factor $S_{\mu(x)} \mathcal{R}_x \otimes S_{\nu(x)} \mathcal{Q}_x^*$ of a summand in (\ref{eq:cauchy}), we apply Bott's Theorem. Namely, consider the corresponding weight
$$\delta(x):=(-\nu(x),\mu(x)),$$
where $-\nu(x)$ has $\gamma_x$ parts and $\mu(x)$ has $\beta_x$ parts (appending with zeroes, if necessary). We have (see \cite[Corollary 4.1.7, Corollary 4.1.9]{jerzy}):

\begin{theorem}\label{thm:bott}
The cohomologies $H^i(\op{Gr}(\beta_x,\alpha_x), S_{\mu(x)} \mathcal{R}_x \otimes S_{\nu(x)} \mathcal{Q}_x^*)$ vanish when $i\neq N(\delta(x))$, and 
$$H^{N(\delta(x))}(\op{Gr}(\beta_x,\alpha_x), S_{\mu(x)} \mathcal{R}_x \otimes S_{\nu(x)} \mathcal{Q}_x^*)=S_{\tau} V_x,$$
where $\tau$ is the non-increasing sequence obtained from $\delta(x)$.
\end{theorem}

\section{Minimal free resolutions of 1-step orbit closures}\label{sec:main}

\vspace{0.05in}

We start with some preliminary lemmas.

\begin{lemma}
\label{lem:first}
Let $\lambda,\mu$ be two partitions and $n$ the number of parts of $\mu$. Then $N(\lambda, \mu+n)\leq |\mu|$. 
\end{lemma}

\begin{proof}
We can assume WLOG that $N(\lambda, \mu+n)\geq 0$. We proceed by induction on $n$. Write $\lambda=(\lambda_1,\dots,\lambda_m)$ and pick $k\in \nat$ to be the largest number with the property $\lambda_{m-k+1}+k\leq \mu_n$ (here set $\lambda_{m+1}:=0$). Then we must have $\lambda_{m-k}+k-1\geq \mu_n$. Applying $n\cdot k$ exchanges we arrive at the sequence
$$(\lambda_1,\dots, \lambda_{m-k},\mu_1+n-k,\mu_2+n-k,\dots, \mu_n+n-k, \lambda_{m-k+1}+n, \lambda_{m-k+2}+n,\dots, \lambda_{m}+n),$$
and we see that the last $k+1$ elements of the sequence are in their final places. We have the following
$$N(\lambda, \mu+n)=n\cdot k + N(\lambda_1,\dots, \lambda_{m-k},\mu_1+n-k,\dots, \mu_n+n-k, \lambda_{m-k+1}+n,\dots, \lambda_{m}+n)=$$
$$=n\cdot k + N(\lambda_1,\dots, \lambda_{m-k},\mu_1+n-k,\dots, \mu_{n-1}+n-k)= $$
$$=n\cdot k + N(\lambda_1+k-1-\mu_n,\dots, \lambda_{m-k}+k-1-\mu_n,\mu_1-\mu_n+n-1,\dots, \mu_{n-1}-\mu_n+n-1).$$
Now the induction hypothesis is satisfied and we obtain 
$$N(\lambda_1+k-1-\mu_n,\dots, \lambda_{m-k}+k-1-\mu_n,\mu_1-\mu_n+n-1,\dots, \mu_{n-1}-\mu_n+n-1)\leq |\mu|-n\cdot \mu_n.$$
Hence we obtain 
$$N(\lambda, \mu+n)\leq |\mu|+n(k-\mu_n) \leq |\mu|.$$
\end{proof}


\begin{lemma}
\label{lem:sec}
For two partitions $\lambda, \mu$ we have $N(-\lambda,\mu)\leq |\lambda^+|+ |\mu^+| + u_\lambda \cdot u_\mu$.
\end{lemma}

\begin{proof}
Put $\lambda=(\lambda_1,\dots,\lambda_m)$, $\mu=(\mu_1,\dots,\mu_n)$ and $u:=u_\lambda,v:=u_\mu$. After $u\cdot v$ exchanges we get
$$ N(-\lambda,\mu)=u \cdot v+N(-\lambda_m,\dots, -\lambda_{u+1},\mu_1-u,\dots, \mu_{v}-u,-\lambda_{u}+v, \dots, -\lambda_1+v, \mu_{v+1},\dots, \mu_n)=$$
$$=u \cdot v+N(-\lambda_m,\dots, -\lambda_{u+1},\mu_1-u,\dots, \mu_{v}-u)+N(-\lambda_{u}+v, \dots, -\lambda_1+v, \mu_{v+1},\dots, \mu_n).$$
Now using Lemma \ref{lem:first} we have
$$N(-\lambda_m,\dots, -\lambda_{u+1},\mu_1-u,\dots, \mu_{v}-u) = N(-\lambda_m+u,\dots, -\lambda_{u+1}+u,\mu_1,\dots, \mu_{v})\leq |\mu^+|.$$
Similarly, we have $N(-\lambda_{u}+v, \dots, -\lambda_1+v, \mu_{v+1},\dots, \mu_n)\leq |\lambda^+|$, hence the conclusion.
\end{proof}


\vspace{0.1in}

The following result is well-known (cf. \cite{fult}):

\begin{lemma}
\label{lem:horn}
Suppose $\nu$ is a partition occurring in the Littlewood-Richardson product of $\lambda$ and $\mu$. Then for any $k\in \nat$
$$\nu_1+\nu_2+\dots+\nu_k\leq \lambda_1+\dots+\lambda_k + \mu_1+\dots+\mu_k.$$
\end{lemma}

\vspace{0.1in}

\begin{definition}\label{def:coll} Let $Q$ be any quiver, and take $\underline{\lambda}$ an associated collection of partitions as in (\ref{eq:coll}). Set
$$D(\underline{\lambda}):= \sum_{a\in Q_1} |\lambda(a)|-\sum_{x\in Q_0} N(-\nu(x),\mu(x)),$$
and define the dimension vector $\underline{u}_{\underline{\lambda}}\in \nat^{Q_0}$ by 
$$(\underline{u}_{\underline{\lambda}})_x:=\max(u_{\mu(x)},u_{\nu(x)}).$$
\end{definition}

\vspace{0.1in}

Using Theorem \ref{thm:bott} and the K\"unneth formula in Section \ref{sec:geom}, we see that a collection $\und{\lambda}$ contributes precisely to the term $F_{D(\und{\lambda})}$ in the complex $F_\bullet$. Now we can state the main result of this section, generalizing \cite[Proposition 3.4]{kavita}:

\begin{theorem}\label{thm:main}
Let $Q$ be any quiver, and $\underline{\lambda}$ an associated collection of partitions. Then
$$D(\underline{\lambda})\geq E_Q(\underline{u}_{\underline{\lambda}}).$$
\end{theorem}

\begin{proof}
Using Lemma \ref{lem:sec} we have
$$D(\lambda)\geq \sum_{a\in Q_1} |\lambda(a)| - \sum_{x\in Q_0} |\nu(x)^+|-|\mu(x)^+|-u_{\mu(x)}\cdot u_{\nu(x)}.$$
For an $x\in Q_0$, we have by Lemma \ref{lem:horn} the inequalities
$$|\mu(x)^+|\leq -u_{\mu(x)}^2+\sum_{a\in Q_1 | ta=x} (|\lambda(a)^+|+u_{\mu(x)}u_{\lambda(a)}),$$
$$|\nu(x)^+|\leq -u_{\nu(x)}^2+\sum_{b\in Q_1 | hb=x} (|\lambda(b)^-|+u_{\lambda(b)}u_{\nu(x)}).$$
Hence we get
$$D(\lambda)\geq \sum_{x\in Q_0} (u_{\mu(x)}^2+u_{\nu(x)}^2-u_{\mu(x)} u_{\nu(x)})+\sum_{a\in Q_1} (u_{\lambda(a)}^2-u_{\mu(ta)}u_{\lambda(a)}-u_{\lambda(a)}u_{\nu(ha)}).$$
From $(u_{\mu(ta)}-u_{\lambda(a)})(u_{\nu(ha)}-u_{\lambda(a)})\geq 0$ we get
\begin{equation}
\label{eq:red} 
D(\lambda)\geq \sum_{x\in Q_0} (u_{\mu(x)}^2+u_{\nu(x)}^2-u_{\mu(x)} u_{\nu(x)})-\sum_{a\in Q_1} u_{\mu(ta)}u_{\nu(ha)}.
\end{equation}
For any $x\in Q_0$, let $A_x$ (resp. $B_x$) be any number with $A_x\geq u_{\mu(x)}$ (resp. $B_x\geq u_{\nu(x)}$). 

Now fix a vertex $x\in Q_0$, and assume $u_{\mu(x)}\geq u_{\nu(x)}$. We show that we have the following inequality
\begin{equation}
\label{eq:red1}
u_{\nu(x)}^2-u_{\mu(x)}u_{\nu(x)}-\sum_{a\in Q_1 | ha=x} A_{ta}u_{\nu(x)} \geq -\sum_{a\in Q_1 | ha=x} A_{ta}u_{\mu(x)},
\end{equation}
which is equivalent to 
$$(u_{\mu(x)}-u_{\nu(x)})(\sum_{a\in Q_1 | ha=x} A_{ta}-u_{\nu(x)})\geq 0.$$
This holds, for
$$\sum_{a\in Q_1 | ha=x} A_{ta}\geq \sum_{a\in Q_1 | ha=x}  u_{\mu(ta)}\geq \sum_{a\in Q_1 | ha=x}  u_{\lambda(a)}\geq u_{\nu(x)},$$
where the last inequality follows from Lemma \ref{lem:horn}. Similarly, in the case $u_{\mu(x)}\leq u_{\nu(x)}$ we have
\begin{equation}
\label{eq:red2}
u_{\mu(x)}^2-u_{\mu(x)}u_{\nu(x)}-\sum_{a\in Q_1 | ta=x} u_{\mu(x)}B_{ha} \geq -\sum_{a\in Q_1 | ta=x} u_{\nu(x)}B_{ha}.
\end{equation}
The proof follows from the successive application over all vertices of the inequalities (\ref{eq:red1}) or (\ref{eq:red2}) to the RHS of (\ref{eq:red}).
\end{proof}

We have the following immediate consequences:

\begin{theorem} Let $Q$ be a Dynkin quiver, and $V$ a 1-step representation. Then $\overline{\orb}_V$ is normal, has rational singularities, and the complex $F_\bullet$ gives the minimal free resolution of its defining ideal.
\end{theorem}

\begin{proof}\label{thm:dynkin}
Since $Q$ is Dynkin, $E_Q$ is a positive definite quadratic form \cite{quad}. We apply Bott's theorem \ref{thm:bott} to compute the non-positive cohomologies in (\ref{eq:cauchy}). Using Theorem \ref{thm:main}, we obtain $F_i=0$ for $i<0$, and $F_0=A$, which proves the claim by Theorem \ref{thm:rational} (b).
\end{proof}

We note that the theorem above generalizes \cite[Theorem 3.5]{kavita},  and  \cite[Theorem 3.5]{lol} as well by Proposition \ref{prop:step3}.  For concrete examples of resolutions, see \cite{kavita2,kavita}.

\begin{theorem}\label{thm:extended}
Let $Q$ be an extended Dynkin quiver, $V$ a 1-step representation and assume that the fiber $q^{-1}(V)$ is connected. Then the normalization of $\overline{\orb}_V$ has rational singularities, and the complex $F_\bullet$ gives the minimal free resolution of the normalization.
\end{theorem}

\begin{proof}
Since $Q$ is extended Dynkin, $E_Q$ is a positive semi-definite quadratic form \cite{quad}. As above, we obtain $F_i=0$ for $i<0$ which proves the claim by part of Theorem \ref{thm:rational} (a).
\end{proof}

We note that in the case of 1-step representations $V$ as in Proposition \ref{prop:step2}, the condition on the fiber in the theorem above is automatically satisfied.

\vspace{0.1in}

First introduced in \cite{scho}, a quiver Grassmannian is a fiber of the map $q$. The geometry of these have been extensively studied recently (e.g. \cite{quivgrass}). The following is immediate from Theorem \ref{thm:dynkin} together with Zariski's Theorem:

\begin{corollary} All quiver Grassmannians of Dynkin quivers are connected varieties.
\end{corollary}

\section{Type $\A$ quivers}\label{sec:type}

Given a quiver $Q$ and a representation $V$ with $\Dim V=\alpha$, it is an interesting problem to find (preferably minimal) generators for the defining ideal $I_V=I(\ove{\orb}_V)$ of the orbit closure $\overline{\orb}_V$ in $\Rep(Q,\alpha)$. For Dynkin quivers, set-theoretic equations are known by \cite{bong} and come from rank conditions. These are indeed defining equations (i.e. they generate a radical ideal) when $Q$ is of the type $\A$ by \cite{lak},\cite{scheme}.

For 1-step orbit closures, the defining ideal $I_V$ can be described in some cases by identifying explicitly the minimal generators in the term $F_1$ of the minimal free resolution. This has been done for the non-equioriented quiver of type $\A_3$ in \cite{kavita2}. We are considering the situation more generally for type $\A$ quivers, and describe minimal generators of 1-step orbit closures explicitly.

Fix $Q$ a type $\A_n$ quiver with underlying Dynkin diagram
\[\xymatrix@M-0.20pc{\stackrel{\disp 1}{\bullet} \ar@{-}[r] & \stackrel{\disp 2}{\bullet} \ar@{-}[r] & \stackrel{\disp 3}{\bullet} \ar@{-}[r] & \cdots \ar@{-}[r] & \stackrel{\disp n-1}{\bullet} \ar@{-}[r] & \stackrel{\disp n}{\bullet} 
}\]
and arbitrary orientation. The orientation determines a sequence $(1=s_1 < s_2 < \dots <s_k-1 < s_k=n)$ of vertices that are sources or sinks. 

For a pair of integers $(p,q)$ with $1\leq p \leq q \leq n$ let $r^{pq}\in\nat^{n}$ denote the positive root  of the Dynkin diagram of $Q$, with $r^{p,q}_i=1$ if $p\leq i\leq q$, and $r^{p,q}_i=0$ otherwise. The indecomposable representations $E_{p,q}$ of $Q$ are in bijection with the positive roots $r^{p,q}$ via $\Dim E_{p,q}=r^{p,q}$.

Let $\alpha$ be a dimension vector for $Q$. For an arrow $a\in Q_1$, we write $X_a$ for the generic matrix of variables of size $\alpha_{ha}\times \alpha_{ta}$. For a sequence of arrows $\bullet\xrightarrow{a_i}\bullet\xrightarrow{a_{i+1}}\dots \xrightarrow{a_{j-1}}\bullet \xrightarrow{a_{j}}\bullet$, we denote the composition by $X_{ta_i,ha_j}=X_{a_j}\circ X_{a_{j-1}} \circ \cdots \circ X_{a_{i+1}}\circ X_{a_i}$.

Fix a non-simple root $r^{p,q}$, where $1\leq p < q \leq n$. We consider the generic matrix $X_{p,q}$ of the linear maps that go from the sources of the support of $r^{p,q}$ to the sinks of the support of $r^{p,q}$. Explicitly, assume WLOG that we have $s_{i-1}\leq p <s_i$ and $s_j < q \leq s_{j+1}$ with $s_{i-1}$ a source and $s_{j+1}$ a sink (the other cases are analogous). Then
\[X_{p,q}=
\begin{bmatrix}
X_{p,s_i} & X_{s_{i+1},s_i} & 0 & \dots & 0 \\
 0 & X_{s_{i+1},s_{i+2}} & X_{s_{i+3},s_{i+2}} & \dots & 0 \\
 0 & 0 & X_{s_{i+3},s_{i+4}} & \dots & 0 \\
 \vdots & \vdots & \vdots & \ddots &\vdots\\
 0 & 0 & 0 & \dots & X_{s_j,q} 
\end{bmatrix}.
\]

Clearly, a \textit{rank condition} $\rank X_{p,q} \leq r$ gives a closed subscheme of $\Rep(Q,\alpha)$ by the $(r+1)\times (r+1)$ minors of $X_{p,q}$. We will give minimal generators of the defining ideals of 1-step orbit closures that are among minors of this type.

There is a more representation-theoretic interpretation of the maps $X_{p,q}$. We denote by $P_x$ the projective cover of the simple module at a vertex $x$. Let $E$ be the representation defined by the cokernel of the natural map between projectives:
\begin{equation}\label{eq:proj}
0 \to \bigoplus_{\substack{y \text{ sink}\\ \text{in supp } r^{p,q}}} P_y \to  \bigoplus_{\substack{x \text{ source}\\ \text{in supp } r^{p,q}}} P_x \to E\to 0.
\end{equation}
It is easy to see that $E$ is indecomposable. Applying $\Hom_Q(-,X)$ to this sequence, we obtain  (as in the sequence (\ref{eq:ringel})) the exact sequence
\begin{equation}\label{eq:hom}
0\to \Hom_Q(E,X)\to \bigoplus_{\substack{x \text{ source}\\ \text{in supp } r^{p,q}}} X_x \xrightarrow{X_{p,q}} \bigoplus_{\substack{y \text{ sink}\\ \text{in supp } r^{p,q}}} X_y \to \Ext_Q(E,X) \to 0.
\end{equation}
Conversely, for any non-projective indecomposable $E$ we can construct a corresponding map $X_{p,q}$ by taking the minimal projective resolution of $E$.

\vspace{0.2in}

Now we define the following quantities inductively:
$$\gamma^{p,s_i}=\min\{\gamma_{p+1},\gamma_{p+2},\dots,\gamma_{s_i}\},\gamma^{s_{i+1},s_i}=\min\{\gamma_{s_{i+1}-1},\dots,\gamma_{s_{i}+1},\gamma_{s_i}-\gamma_p\},$$
\[\gamma^{s_{i+1},s_{i+2}}=\!\min\{\gamma_{s_{i+1}+1},\dots,\gamma_{s_{i+2}}\},\gamma^{s_{i+3},s_{i+2}}=\!\min\{\gamma_{s_{i+3}-1},\dots,\gamma_{s_{i+2}+1},\gamma_{s_{i+2}}-\gamma_{s_{i+1}}+\gamma^{s_{i+1},s_i}\},\]
$$\dots,\gamma^{s_j,s_{j-1}}=\min\{\gamma_{s_j-1},\dots,\gamma_{s_{j-1}+1},\gamma_{s_{j-1}}-\gamma_{s_{j-2}}+\gamma^{s_{j-2},s_{j-3}}\},\gamma^{s_j,q}=\min\{\gamma_{s_j+1},\dots,\gamma_{q}\},$$
and similarly,
\[\beta^{p,s_i}=\min\{\beta_p,\dots, \beta_{{s_i}-1}\}, \beta^{s_{i+1},s_i} =  \min\{\beta_{s_{i+1}},\beta_{s_{i+1}-1},\dots,\beta_{s_i+1}\},\]
\[\beta^{s_{i+1},s_{i+2}} = \min\{\beta_{s_{i+1}}-\beta_{s_{i}}+\beta^{p,s_i},\beta_{s_{i+1}+1},\dots,\beta_{s_{i+2}-1}\}, \beta^{s_{i+3},s_{i+2}} = \min\{\beta_{s_{i+3}},\dots, \beta_{s_{i+2}+1}\},\]
\[\dots,\beta^{s_j,s_{j-1}} = \min\{\beta_{s_{j}},\dots,\beta_{s_{j-1}+1} \},\beta^{s_j,q}=\min\{\beta_{s_{j}}-\beta_{s_{j-1}}+\beta^{s_{j-2},s_{j-1}},\beta_{s_{j}+1},\dots,\beta_{q-1}\}.\]


\vspace{0.1in}

Let  $e=j-i+2$, which is the number of \lq\lq equioriented parts'' of the support of $r^{p,q}$. Denote by $B_{p,q}$ the set of all sequences $RC$ of pairs of non-negative integers 
$$RC=((R_1,C_1),(R_2,C_2),\dots,(R_e,C_e)),$$
satisfying the following inequalities and equations:

$$R_1<\gamma^{p,s_i},\, R_2<\gamma^{s_i,s_{i+1}},\dots, R_e < \gamma^{s_j,q}\, \text{ and }\, C_1<\beta^{p,s_i}, \, C_2<\beta^{s_i,s_{i+1}},\dots,C_e<\beta^{s_j,q},$$
$$R_1 = \gamma_p\,, \,R_1+R_2 < \gamma_{s_i}\,, \,R_2+R_3=\gamma_{s_{i+1}}-1, \dots , R_{e-1}+R_{e}=\gamma_{s_{j}}-1,\hspace{0.3in}$$
$$\hspace{0.5in}  C_1+C_2 = \beta_{s_i}-1\, , \, C_2+C_3<\beta_{s_{i+1}},\,\,\dots\,\,,\,\, C_{e-1}+C_{e} < \beta_{s_j}\,,\,\,\, C_e=\beta_q.$$

\vspace{0.1in}

\begin{definition}\label{def:relevant}
For $1\leq p < q \leq n$ as above,  we call $r^{p,q}$ \textit{relevant} if the following inequalities hold:
$$\gamma_p<\gamma^{p,s_i}\,,\,\gamma_{s_{i+1}}<\gamma^{s_{i+1},s_{i}}+\gamma^{s_{i+1},s_{i+2}},\dots, \gamma_{s_j}<\gamma^{s_{j},s_{j-1}}+\gamma^{s_{j},q}, $$
\[\beta_{s_i}<\beta^{p,s_{i}}+\beta^{s_{i+1},s_{i}}, \dots, ,\beta_{s_{j-1}}<\beta^{s_{j-2},s_{j-1}}+\beta^{s_{j},s_{j-1}}\,,\, \beta_q < \beta^{s_j,q}.\]
\end{definition}

The inequalities are analogous for the other cases (i.e. $p$ source and $q$ source, $p$ sink and $q$ source, $p$ sink and $q$ sink in the support of $r^{p,q}$). It is easy to see that the set $B_{p,q}$ is non-empty if and only if $r^{p,q}$ is relevant. For illustration, we consider the following example:

\begin{example}\label{ex:main}
Consider the following quiver of type $\A_7$.
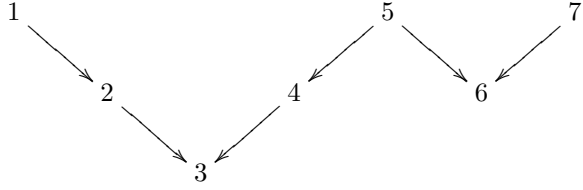
\[\xymatrix@R=1.5pc{
1 \ar[dr] & & & & 5 \ar[dl]\ar[dr] & &7 \ar[dl]\\
& 2 \ar[dr] & & 4 \ar[dl] & & 6 &\\
& & 3 & & & &
}\]

Suppose $\alpha=\gamma+\beta$, where $\gamma,\beta$ are the following dimension vectors:
\[\gamma = (2,4,3,2,1,1,0) \mbox{ and } \beta = (1,1,1,1,1,2,2)\]
In the notation above, we have $p=1,s_i=3,s_{i+1}=5,s_{i+2}=6, q =7$ and $e=4$. We show that the root $r^{1,7}$ is relevant. We have:
\[\gamma^{1,3} = \min\{\gamma_2,\gamma_3\}=3 \, , \, \gamma^{5,3}=\min\{\gamma_4,\gamma_3-\gamma_1\}=1,\]
\[\gamma^{5,6} = \min\{\gamma_6\}=1 \,,\,\gamma^{7,6} = \min\{\gamma_6-\gamma_5+\gamma^{5,3}\}=1.\]
Clearly, the inequalities as in Definition \ref{def:relevant} involving $\gamma$ are satisfied:
\[\gamma_1 = 2 < 3 = \gamma^{1,3} \, , \, \gamma_5 =1 < 2 =\gamma^{5,3}+\gamma^{5,6} \, , \, \gamma_7 = 0 < 1= \gamma^{7,6}.\]
Similarly, we have
\[\beta^{1,3} = \min\{ \beta_1,\beta_2\} = 1 \,,\, \beta^{5,3} = \min\{\beta_5,\beta_4\} = 1, \]
\[\beta^{5,6} = \min\{\beta_5-\beta_3+\beta^{1,3}\} = 1\,,\, \beta^{7,6}=\min\{\beta_7\} = 2.\]
The inequalities as in Definition \ref{def:relevant} involving $\beta$ are satisfied as well:
\[\beta_3 = 1 <  2 = \beta^{1,3}+\beta^{5,3}\,,\, \beta_6 = 2 < 3 = \beta^{5,6}+\beta^{7,6}.\]
Hence $r^{p,q}$ is relevant. Moreover, we can see easily that there is only one tuple $RC\in B_{1,7}$:
\[((2,0),(0,0),(0,0),(0,1)).\]

\end{example}

Now we can formulate the main theorem of this section:

\begin{theorem}\label{thm:min}
Let $Q$ be a type $\A$ Dynkin quiver and $V \in \Rep(Q,\alpha)$ a 1-step representation with
$$q:Z(\beta \subset \alpha)\twoheadrightarrow  \overline{\orb}_V,$$
and write $\alpha=\beta+\gamma$. Then minimal generators of the defining ideal of $\ove{\orb}_V$ are given by
\[F_1=\!\!\!\!\bigoplus_{\substack{r^{p,q} \\ \text{relevant}}} \bigoplus_{RC\in B_{p,q}}\!\!\!\! \bigwedge^{\gamma_p+C_1+1}\!\!\! V_{p}\otimes \!\!\!\!\! \bigwedge^{{R_1+R_2+\beta_{s_i}+1}}\!\!\!\! V_{s_i}^*\otimes \!\!\!\!\! \bigwedge^{\gamma_{s_{i+1}}+C_2+C_3+1}\!\!\!\!\!\!\!\!\!\!\!\!\!V_{s_{i+1}}\otimes\cdots\otimes\!\!\!\!\!\!\!\!\!\!\! \bigwedge^{\gamma_{s_j}+C_{e-1}+C_e+1}\!\!\!\!\! V_{s_j} \otimes\!\!\!\! \bigwedge^{R_e+\beta_q+1}\!\!\!V_{q}^*\otimes A(-d_{RC}). \]
In other words,  for each relevant $r^{p,q}$, a collection $RC\in B_{p,q}$ gives generators that are $(m+1) \times (m+1)$ minors of $X_{p,q}$ (see Remark \ref{rem:minors}) of some degree $d_{RC}$, where 
\[m=1+\ds_{\substack{x \text{ source}\\ \text{in supp } r^{p,q}}} \gamma_x +\ds_{\substack{y \text{ sink}\\ \text{in supp } r^{p,q}}} \beta_y.\]
\end{theorem}

\begin{proof}
We use the notation from the previous sections. To find the defining equations, we compute explicitly the first term $F_1$ from the minimal free resolution of $\overline{\orb}_V$. Let $\underline{\lambda}$ be a collection of partitions as in (\ref{eq:coll}), and assume $D(\underline{\lambda})=1$ (i.e. $\underline{\lambda}$ contributes to $F_1$, see Definition \ref{def:coll}). By Theorem \ref{thm:main}, $\und{u}_{\und{\lambda}}$ must be a root, say $\und{u}_{\und{\lambda}}=r^{p,q}$, for some $1\leq p \leq q \leq n$. Clearly, this implies by the construction of $\und{u}_{\und{\lambda}}$ that there is at least one arrow $a\in Q_1$ such that the Durfee size of $\lambda(a)$ is $u_{\lambda(a)}=1$. Then $(\und{u}_{\und{\lambda}})_{ta}=(\und{u}_{\und{\lambda}})_{ha}=1$, hence $r^{p,q}$ is not a simple root, i.e. $p\neq q$. 

Let $s_{i-1}\leq p <s_i$ and $s_j < q \leq s_{j+1}$, and assume WLOG that $s_{i-1}$ a source and $s_{j+1}$ a sink. 

Let $Q'$ denote the subquiver of $Q$ that supports $r^{p,q}$. We have the following by construction:
\begin{itemize}
\item If $a\in Q'_1$ with $ta=p$ (resp. $ha=q$), then we have that $\mu(p)=\lambda(a)$ and $\nu(p)$ is trivial (resp. $\nu(q)=\lambda(a)$ and $\mu(q)$ is trivial).
\item If $x\in Q'_0$ is not source/sink with $\xrightarrow{b} x \xrightarrow{a}$, then we have $\mu(x)=\lambda(a)$ and $\nu(x)=\lambda(b)$.
\item If $x=s_l\in Q'_0$ with $i\leq l\leq j$ is a source $\xleftarrow{a_1} x \xrightarrow{a_2}$ (resp. sink $\xrightarrow{b_1} x \xleftarrow{b_2}$), then $\mu(x)$ is in the Littlewood-Richardson product of $\lambda(a_1)$ and $\lambda(a_2)$, and $\nu(x)$ is trivial (resp. $\nu(x)$ is in the L.-R. product of $\lambda(b_1)$ and $\lambda(b_2)$, and $\mu(x)$ is trivial).
\end{itemize}
Clearly, if $a\notin Q'_1$, then $u_{\lambda(a)}=0$, hence $\lambda(a)$ is trivial, and if $a\in Q'_1$, then $u_{\lambda(a)}\leq 1$. In fact, we show that in the latter case we must have $u_{\lambda(a)}=1$. This is clear if $ta=p$ or $ha=q$, so take $a\in Q'_1$ not of this type and assume by contradiction that $\lambda(a)$ is trivial. Removing the arrow $a$ from $Q'$ we obtain two connected components, say $Q'^1$ and $Q'^2$. Since $\lambda(a)$ is trivial, we can divide the collection of partitions $\und{\lambda}$ into two non-trivial collections $\und{\lambda}^1,\und{\lambda}^2$ on $Q'^1,Q'^2$, respectively. But then using Theorem \ref{thm:main} we get $1=D(\und{\lambda})= D(\und{\lambda}^1)+D(\und{\lambda}^2)\geq 1+1$, a contradiction. 

\vspace{0.05in}

Hence for any $a\in Q'_1$, $\lambda(a)$ is a hook, say $\lambda(a)=(r(a)+1,1^{c(a)})$, with $r(a)+1\leq \gamma_{ha}$ and $c(a)+1\leq \beta_{ta}$. Similarly, if $x$ is a source (resp. sink) in $Q'$, we put $\mu(x)=(r(x)+1,1^{c(x)})$ with $c(x)+1\leq \beta_x$ (resp. $\nu(x)=(r(x)+1,1^{c(x)})$ with $c(x)+1\leq \gamma_x$).

\vspace{0.05in}

Since we have equality $D(\underline{\lambda})=E_Q(\underline{u}_{\underline{\lambda}})$ in Theorem \ref{thm:main}, each inequality used in its proof must be an equality. It is easy to see that we have equality already in (\ref{eq:red}). Next, we must have equality in Lemma \ref{lem:horn} for $k=1$ whenever it is used. This implies that if $x=s_l\in Q'_0$ with $i\leq l\leq j$ is a source $\xleftarrow{a_1} x \xrightarrow{a_2}$ (resp. sink $\xrightarrow{b_1} x \xleftarrow{b_2}$) then $r(x)= r(a_1)+r(a_2)+1$ and, consequently $c(x)=c(a_1)+c(a_2)$ (resp. $r(x)=c(b_1)+c(b_2)+1$ and $c(x)=r(b_1)+r(b_2)$\,). Notice that such $\mu(x)$ (resp. $\nu(x)$) appears indeed with multiplicity 1 in the L.-R. product of $\lambda(a_1)$ and $\lambda(a_2)$ (resp. $\lambda(b_1)'$ and $\lambda(b_2)'$). Now we are left so see the implications of equality in Lemma \ref{lem:sec} whenever it is used in the proof of Theorem \ref{thm:main}. We check the implications case by case:
\begin{itemize} 
\item If $x$ is a source (resp. sink) in $Q'$ then we get $r(x)=\gamma_x$ (resp. $r(x)=\beta_x$). Moreover, after performing the exchanges we arrive to the sequence $\tau(x)=(1^{\gamma_x+c(x)+1})$ (resp. $\tau(q)=(-1^{\beta_q+c(x)+1}$)\,).
\item If $x\in Q'_0$ is not source/sink, say $\xrightarrow{b} x \xrightarrow{a}$, then we get $\mu(x)=\nu(x)'$, i.e. $\lambda(b)=\lambda(a)$. Moreover, after performing the exchanges we arrive to the trivial partition.
\end{itemize}
Recall that $e=j-i+2$. Now choose an arrow $a_1,a_2,\dots,a_e$ between $p$ and $s_i$, $s_i$ and $s_{i+1}$, \dots, $s_j$ and $s_{q}$, respectively. Then we put $R_i = r(a_i)$ and $C_i=c(a_i)$. Putting the inequalities obtained together we see that $((R_1,C_1),(R_2,C_2),\dots,(R_e,C_e))\in B_{p,q}$. Moreover the collection $\und{\lambda}$ contributes to $F_1$ with the cohomology:

 $$\bigwedge^{\gamma_p+C_1+1}\!\! V_{p}\otimes \!\!\!\! \bigwedge^{{R_1+R_2+\beta_{s_i}+1}}\!\!\! V_{s_i}^*\otimes \!\!\!\! \bigwedge^{\gamma_{s_{i+1}}+C_2+C_3+1}\!\!\!V_{s_{i+1}}\otimes\cdots\otimes\!\!\!\!\! \bigwedge^{\gamma_{s_j}+C_{e-1}+C_e+1}\!\!\! V_{s_j} \otimes\!\! \bigwedge^{R_e+\beta_q+1}\!\!V_{q}^*.$$

Using Cauchy's formula (see \cite{jerzy}), it is easy to see that the multiplicity of this representation in the coordinate ring $A=\Sym (\oplus_{a\in Q_1} V_{ta}\otimes V^*_{ha})$ is $1$, and it is spanned by minors of $X_{p,q}$ (see Remark \ref{rem:minors}). Also, their degree is 
\[d_{RC}=\ds_{a \in Q'_1} |\lambda(a)|=\ds_{t=1}^{e} N_t(R_t +C_t+1),\]
where $N_1,N_2,\dots, N_e$ are the number of arrows between  $p$ and $s_i$, $s_i$ and $s_{i+1}$, \dots, $s_j$ and $s_{q}$, respectively.
\end{proof}

\begin{remark}\label{rem:minors}
We can describe explicitly the minimal generating minors spanning $F_1$ as in Theorem \ref{thm:min}.  The matrix 
\[X_{p,q}=
\begin{bmatrix}
X_{p,s_i} & X_{s_{i+1},s_i} & 0 & \dots & 0 \\
 0 & X_{s_{i+1},s_{i+2}} & X_{s_{i+3},s_{i+2}} & \dots & 0 \\
 0 & 0 & X_{s_{i+3},s_{i+4}} & \dots & 0 \\
 \vdots & \vdots & \vdots & \ddots &\vdots\\
 0 & 0 & 0 & \dots & X_{s_j,q} 
\end{bmatrix}.
\]
is formed by the obvious blocks. We divide the rows and columns into blocks accordingly. Then a collection $RC=((R_1,C_1),(R_2,C_2),\dots,(R_e,C_e))$ gives precisely those minors that we get by choosing $\gamma_p+C_1+1$ columns of the first block, $R_1+R_2+\beta_{s_i}+1$ rows of the first block, $\gamma_{s_{i+1}}+C_2+C_3+1$ columns of the second block etc.
\end{remark}

\begin{example}
We consider the $\A_7$ quiver as in Example \ref{ex:main}. We saw that $r^{1,7}$ is relevant, and it gives only one tuple $((2,0),(0,0),(0,0),(0,1))$. This tuple gives in Theorem \ref{thm:min} precisely the determinant of the matrix
\[\begin{bmatrix}
X_{1,3} & X_{5,3} & 0 \\
 0 & X_{5,6} & X_{7,6}
\end{bmatrix}.\]
\end{example}

Recall that $T_\beta\in\Rep(Q,\beta)$ and $T_\gamma\in\Rep(Q,\gamma)$ denote the generic representations. Using the notation above, we get:

\begin{theorem}\label{thm:rank}
The rank conditions 
$$\{\rank X_{p,q}\leq \ds_{\substack{x \text{ source}\\ \text{in supp } r^{p,q}}} \gamma_x +\ds_{\substack{y \text{ sink}\\ \text{in supp } r^{p,q}}} \beta_y\,\, | \text{ where } \, p<q \text{ such that } r^{p,q} \text{ relevant}\}$$
define the 1-step orbit closure $\ove{\orb}_V$ in $\Rep(Q,\alpha)$. Moreover, $X\in \ove{\orb}_V$ if and only if we have
$$\dim \Hom_Q(E,X)\geq \langle \Dim E,\beta\rangle,$$
for all indecomposable $E$ that satisfy: 
\begin{itemize}
\item[$\bullet$] $\Hom_Q(E,T_\gamma)=0$, and for all (non-zero) indecomposable quotient modules $E'$ of $E$, we have $\Ext_Q(E',T_\gamma)\neq 0$, and
\item[$\bullet$] $\Ext_Q(E,T_\beta)=0$, and for all (non-zero) indecomposable submodules $E''$ of $E$, we have $\Hom_Q(E'',T_\beta)\neq 0$.
\end{itemize}
\end{theorem}

\begin{proof}
Fix $r^{p,q}$ a relevant root. Let $E$ be the representation defined as in (\ref{eq:proj}). We want to show that $E$ satisfies all the listed properties. We check only the first claim, as the other can be showed analogously. Note that by (\ref{eq:hom}),  $\Hom_Q(E,T_\gamma)=0$ is equivalent to $(T_{\gamma})_{p,q}$ being injective. Now let $E'$ be an arbitrary (non-zero) indecomposable quotient representation of $E$. Then taking a minimal projective resolution of $E'$ (as in (\ref{eq:proj})) and applying $\Hom_Q(-,X)$, we obtain (as in (\ref{eq:hom})) another map $X_{p',q'}$ with $p\leq p' \leq q' \leq q$ and the property that the sources in the support of $r^{p',q'}$ are among the sources of the support of $r^{p,q}$ (moreover, any such map can be obtained in this way from a quotient of $E$). As in (\ref{eq:hom}), we see that  $\Ext_Q(E',T_\gamma)\neq 0$ if and only if $T_{p',q'}$ is not surjective. So all the claimed properties of $E$ depend only on the support of $r^{p,q}$.


So we can assume that $Q$ is equal to the support of $r^{p,q}$. We use the same notation for the restrictions of the representations $T_\lambda,T_\beta$ to this support (they remain generic). WLOG, we assume that $q$ is a source. First, assume that $p$ is a source as well.  Applying the construction in (\ref{eq:proj}) to this quiver, we obtain $E=E_{p,q}$. Let $E'$ be a quotient of $E_{p,q}$ corresponding to $X_{p',q'}$ as in (\ref{eq:hom}). In order to show that $T_{\gamma})_{p,q}$ is injective and $T_{p',q'}$ is not surjective, by (\ref{eq:hom}) we must show that $\Hom_Q(E_{p,q},T_{\gamma})=0$ and $\Ext_Q(E',T_{\gamma})\neq 0$. 





We recall the notion of reflection functors (for more, see \cite{elements}): when applied to a source $x\leftarrow s \rightarrow y$, the reflection changes the quiver by reversing all arrows going to $s$, so: $x\rightarrow s \leftarrow y$. When applied to a generic representation $T_\gamma$ of the original quiver, the reflection functor $R^s$ gives a new generic representation $R^s(T_\gamma)=T_{\gamma'}$ with $\gamma'_s=\gamma_x+\gamma_y-\gamma_s$ whenever this is non-negative (at the other vertices $\gamma$ does not change). Moreover, $R^s$ does not change the dimension of the morphisms and extensions between generic representations.

As mentioned, we work on the support of $r^{p,q}$. First, we show that $\Hom_Q(E_{p,q},T_{\gamma})=0$. We apply reflections successively at the vertices $p,p+1,\dots ,s_{i}-1$. Since $\gamma_p<\gamma^{p,s_i}$, the generic representation $T_\gamma$ will be sent to another generic representation $T_{\gamma'}$, and $\gamma'_{s_{i}-1}=\gamma_{s_i}-\gamma_p$. The indecomposable $E_{p,q}$ will be sent to $E_{s_i,q}$, and $\dim \Hom_Q(E_{p,q},T_\gamma)=\dim\Hom_Q(E_{s_i,q},T_{\gamma'})$. Since $s_{i}-1$ is a sink, we can delete the vertices $p,p+1,\dots s_{i}-2$, and not change the dimension of morphisms. Moreover, it is easy to see that $r^{s_i-1,q}$ is relevant with respect to $\gamma'$, that is, the inequalities involving $\gamma'$ necessary for $r^{s_i-1,q}$ to be relevant are satisfied (see Definition \ref{def:relevant}). Applying the construction (\ref{eq:proj}) for this smaller quiver, we get precisely $E=E_{s_i,q-1}$. So we reduced the argument to the case when the quiver starts with a sink. More precisely, we can assume from the onset $p$ is a sink in the support of $r^{p,q}$ with $ p<s_i$, where $s_i$ a source, $\gamma$ a dimension vector and $r^{p,q}$ is relevant with respect to $\gamma$. In the construction (\ref{eq:proj}), we have $E=E_{p+1,q}$. We want to prove that $\dim\Hom_Q(E_{p+1,q},T_\gamma)=0$. We proceed by induction on the number of vertices. If $s_i=q$, then the quiver is equioriented in which case the result is easy to check directly. Otherwise, we reflect successively on the vertices $s_i,s_{i-1},\dots, p+1$ and obtain a dimension vector $\gamma''$. Since $\gamma_{s_{i}}<\gamma^{s_{i},s_{p}}+\gamma^{s_{i},s_{i+1}}$, we see that $\gamma''$ has positive entries, in particular $T_\gamma$ is sent to a generic representation $T_{\gamma''}$. Also, $E_{p+1,q}$ is sent to $E_{p+2,q}$, so we can delete the vertex $p$ without changing the dimension of morphisms. Moreover, it is easy to see that $r^{p+1,q}$ is relevant with respect to $\gamma''$. Hence we conclude by induction that $\Hom_Q(E_{p,q},T_{\gamma})=0$. Now to see that $\Ext_Q(E',T_{\gamma})\neq 0$, we note that the equioriented case is easy to see directly. In the induction process above we always worked with roots relevant with respect to $\gamma$, and the reflection functor applied to a source preserves the dimensions of extensions whenever $E'$ is not a simple representation at a source. Similarly, it is easy to see that the deleted vertices in the process did not change $\Ext(E',T_\gamma)$. So we only need to check that if $E'$ is a simple representation at a source and the inequalities for $\gamma$ as in Definition \ref{def:relevant} are satisfied, then $\Ext(E',T_\gamma)\neq 0$. This is easy to see directly.

Now let $Q$ denote our original quiver. If $r^{p,q}$ is relevant, and $E$ is obtained by the construction in (\ref{eq:proj}), we showed that $\Hom_Q(E,T_\gamma)=0$ and $\Ext_Q(E,T_\beta)=0$. Applying the functor $\Hom_Q(E,-)$ to the exact sequence $0\to T_\beta \to V \to T_\gamma\to 0,$ we obtain that $\dim\Hom_Q(E,V)=\dim\Hom_Q(E,T_\beta)=\langle \Dim E,\beta\rangle$. Using this, by (\ref{eq:hom}) we get for $r^{p,q}$ relevant that
\[\rank V_{p,q} = \ds_{\substack{x \text{ source}\\ \text{in supp } r^{p,q}}} \gamma_x +\ds_{\substack{y \text{ sink}\\ \text{in supp } r^{p,q}}} \beta_y.\]
Using Theorem \ref{thm:min}, we obtain the conclusion.
\end{proof}

For any representation $Y$ of a quiver of type $\A_n$, it is known (see \cite{abeasis2}) that $X\in \ove{\orb}_Y$ if and only if $\rank X_{p,q}\leq \rank Y_{p,q}$, for all $1\leq p<q \leq n$. In fact, the rank conditions give defining equations of $\ove{\orb}_Y$ by \cite[Theorem 6.4]{scheme}. Using the 1-step representations, we give an alternative proof of this fact. Moreover, the proof we give is an effective algorithm for producing a \textit{smaller} set of minors coming from rank conditions that generate the defining ideal of the orbit closure of a representation:

\begin{theorem}\label{thm:scheme}
Let $Q$ be a quiver of type $\A_n$ and $Y\in \Rep(Q,\alpha)$ be an arbitrary representation. Then $\ove{\orb}_Y$ can be written as a (scheme-theoretic) intersection of $n-1$ one-step orbit closures. In particular, some rank conditions generate the defining ideal of $\ove{\orb}_Y$.
\end{theorem}

\begin{proof}
First, we show that such a set-theoretic intersection exists. Fix a vertex $p\in \{1,2,\dots, n-1\}$. We construct a 1-step representation $W^p \in \Rep(Q,\alpha)$ such that $Y \in \ove{\orb}_{W^p}$ and $\rank W^p_{p,q} = \rank Y_{p,q}$, for any $q \in \{p+1, \dots , n\}$. As usual, we assume WLOG that $s_{i-1}\leq p < s_i$ with $s_{i-1}$ a source. We construct a subrepresentation $Z^p$ of $Y$ as follows. First, if $x\leq p$, put $Z^p_x=Y_x$. In the support of $r^{p,n}$ the subrepresentation $Z^p$ of $Y$ looks as follows: 
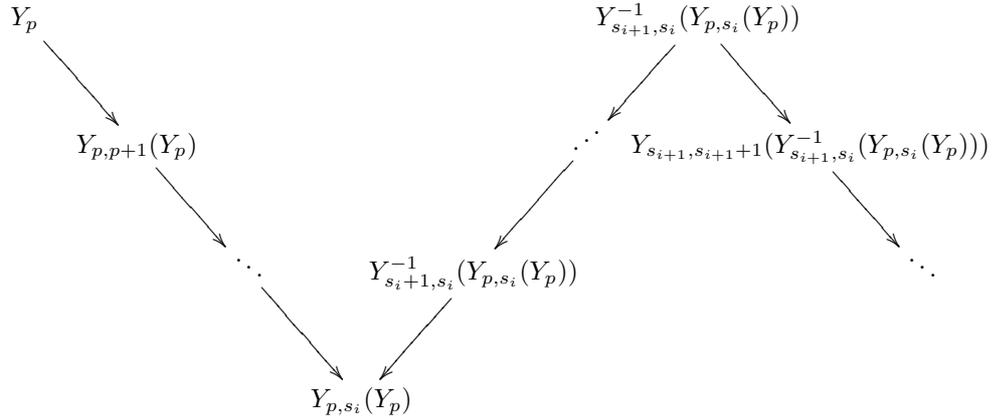
\[\xymatrix@!0@C=3.5pc@R=4pc{
Y_p \ar[dr] &  & & & & & Y^{-1}_{s_{i+1},s_i}(Y_{p,s_i}(Y_p))\ar[dl] \ar[dr] & &\\
      & Y_{p,p+1}(Y_p) \ar[dr] & & &  & \iddots \ar[dl] & & Y_{s_{i+1},s_{i+1}+1}(Y^{-1}_{s_{i+1},s_i}(Y_{p,s_i}(Y_p))) \ar[dr] &\\
      & & \ddots\ar[dr] & & Y^{-1}_{s_i+1,s_i}(Y_{p,s_i}(Y_p))\ar[dl] & & & & \ddots\\
      & & & Y_{p,s_i}(Y_p) & & & & &
}\]
Here the maps between the vector spaces are just the restrictions of the maps of $Y$. Now put $\beta^p = \Dim Z^p$ and consider the 1-step representation $W^p$ obtained by
\[ Z(\beta \subset \alpha) \twoheadrightarrow \ove{\orb}_{W^p}.\]
By construction, $Y\in  \ove{\orb}_{W^p}$. Pick any $q \in \{p+1, \dots , n\}$. Let $E$ the indecomposable associated to the map $X_{p,q}$ as in (\ref{eq:proj}) and (\ref{eq:hom}).  We want to show that $\rank W^p_{p,q} = \rank Y_{p,q}$, or equivalently $\Hom_Q(E, W^p) = \Hom_Q(E,Y)$. Since  $Y\in  \ove{\orb}_{W^p}$, we have  $\dim \Hom_Q(E, W^p) \leq \dim \Hom_Q(E,Y)$.

It is easy to see directly that we have $\Hom_Q(E,Y/(Z^p))=0$, hence $\dim \Hom_Q(E,Y)=\dim \Hom_Q(E,Z^p)$.

Let $T_\beta\in \Rep(Q,\beta)$ be the generic representation. Since $T_\beta \subset W^p$, we get $\dim \Hom_Q(E,T_\beta)\! \leq \dim \Hom_Q(E, W^p)$. 
Note that $Z_{p,q}^p$ is surjective, hence by (\ref{eq:hom}) we have $\Ext(E,Z^p)=0$. Since $T_\beta$ is generic, we must have $\Ext(E,T_\beta)=0$ as well, hence by (\ref{eq:hom}) again $\dim \Hom_Q(E,Z^p) = \dim \Hom_Q(E,T_\beta)$. So we obtained
\[\dim \Hom_Q(E,Y) = \dim \Hom_Q(E,Z^p) = \dim \Hom_Q(E,T_\beta) \leq \dim \Hom_Q(E,W^p),\]
showing that $\rank W^p_{p,q} = \rank Y_{p,q}$.

Now consider the intersection of $n-1$ one-step orbit closures
\[\mathcal{X}=\bigcap_{p=1}^{n-1} \ove{\orb}_{W^p}.\]
Clearly, $\ove{\orb}_Y \subset \mathcal{X}$. In order to see the reverse inclusion, pick any $M \in \mathcal{X}$. Then for any $p,q$ with $1\leq p<q \leq n$, we must have $\rank M \leq \rank W^p_{p,q} = \rank Y_{p,q}$. By \cite{abeasis2}, this implies $M\in \ove{\orb}_Y$. Hence $\mathcal{X}=\ove{\orb}_Y$.

Now we want to see that the intersection is in fact scheme-theoretic. We note that the defining equations of a 1-step orbit closure are defined over $\integ$ (see Remark \ref{rem:minors}). Reducing this scheme modulo $p$ for any prime $p\gg 0$, we obtain a geometrically reduced 1-step orbit closure.  By \cite{ryan}, over a perfect field of characteristic $p$ there is a Frobenius splitting on $\Rep(Q,\alpha)$ that compatibly splits all orbit closures. In particular, the intersection of (1-step) orbit closures is reduced over such a field by \cite[Proposition 1.2.1]{briku}. We can lift this result to characteristic $0$ (as in \cite[Corollary 1.6.6]{briku}), finishing the proof.
\end{proof}

\begin{remark} The argument above using Frobenius splitting fails for other Dynkin types. The first author constructed in \cite{lol} an example of a nullcone of a type $\mathbb{E}_8$ quiver that is not reduced. But this nullcone is the intersection of the codimension 1 orbit closures (which are in fact subgeneric, hence 1-step by Proposition \ref{prop:step3}).
\end{remark}

\bibliography{biblo}

\providecommand{\bysame}{\leavevmode\hbox to3em{\hrulefill}\thinspace}
\providecommand{\MR}{\relax\ifhmode\unskip\space\fi MR }
\providecommand{\MRhref}[2]{%
  \href{http://www.ams.org/mathscinet-getitem?mr=#1}{#2}
}
\providecommand{\href}[2]{#2}
\begin{thebibliography}{10}

\bibitem{abeasis}
S.~Abeasis, \emph{{C}odimension 1 orbits and semi-invariants for the
  representations of an oriented graph of type $\mathcal{A}_n$}, Trans. Amer.
  Math. Soc. \textbf{282} (1984), no.~2, 463--485.

\bibitem{abeasis2}
S.~Abeasis and A.~Del Fra, \emph{Degenerations for the representations of a
  quiver of type {$\mathcal{A}_m$}}, J. Algebra \textbf{93} (1985), 376--412.

\bibitem{abea}
S.~Abeasis, A.~Del Fra, and H.~Kraft, \emph{The geometry of representations of
  {$A_m$}}, Mathematische Annalen \textbf{256} (1981), no.~3, 401--418.

\bibitem{elements}
I.~Assem, D.~Simson, and A.~Skowro{\'n}ski, \emph{Elements of the
  representation theory of associative algebras, {V}ol. 1, {T}echniques of
  representation theory}, London Mathematical Society Student Texts, vol.~65,
  Cambridge University Press, Cambridge, 2006.

\bibitem{orb1}
G.~Bobi\'{n}ski and G.~Zwara, \emph{Normality of orbit closures for {D}ynkin
  quivers of type {$\mathbb{A}_n$}}, Manuscripta Math. \textbf{105} (2001),
  103--109.

\bibitem{orb2}
\bysame, \emph{Schubert varieties and representations of {D}ynkin quivers},
  Colloq. Math. \textbf{94} (2002), 285--309.

\bibitem{bong}
K.~Bongartz, \emph{On degenerations and extensions of finite dimensional
  modules}, Adv. Math. \textbf{121} (1996), 245--287.

\bibitem{briku}
M.~Brion and S.~Kumar, \emph{Frobenius splitting methods in geometry and
  representation theory}, Progress in Mathematics, vol. 231, Birkh\"auser
  Boston, Inc., Boston, MA.

\bibitem{canon}
H.~Derksen and J.~Weyman, \emph{On the canonical decomposition of quiver
  representations}, Compos. Math. \textbf{133} (2002), 245--265.

\bibitem{fult}
W.~Fulton, \emph{Eigenvalues of sums of {H}ermitian matrices (after {A}.
  {K}lyachko)}, S\'eminaire Bourbaki. Vol. 1997/98, Ast\'erisque (1998),
  no.~252, 255--269, Exp. 845.

\bibitem{quivgrass}
G.~C. Irelli, E.~Feigin, and M.~Reineke, \emph{Quiver {G}rassmannians and
  degenerate flag varieties}, Algebra Number Theory \textbf{6} (2012), no.~1,
  165--194.

\bibitem{ryan}
R.~Kinser and J.~Rajchgot, \emph{Type {A} quiver loci and {S}chubert
  varieties}, J. Commut. Algebra \textbf{7} (2015), no.~2, 265--301.

\bibitem{kovacs}
S.~Kov\'acs, \emph{A characterization of rational singularities}, Duke Math. J.
  \textbf{102} (2000), no.~2, 187--191.

\bibitem{lak}
V.~Lakshmibai and P.~Magyar, \emph{Degeneracy schemes, quiver schemes, and
  {S}chubert varieties}, Int. Res. Res. Not. \textbf{1998} (1998), no.~12,
  627--640.

\bibitem{lasc}
A.~Lascoux, \emph{Syzygies des vari\'et\'es d\'eterminantales}, Adv. Math.
  \textbf{30} (1978), no.~3, 202--237.

\bibitem{lol}
A.~C. L\H{o}rincz, \emph{Singularities of zero sets of semi-invariants for
  quivers}, \url{http://arxiv.org/pdf/1509.04170v2.pdf}.

\bibitem{bfun}
\bysame, \emph{The $b$-functions of semi-invariants of quivers}, J. Algebra
  \textbf{482} (2017), 346--–363.

\bibitem{generic}
M.~Reineke, \emph{Generic extensions and multiplicative bases of quantum groups
  at $q=0$}, Represent. Theory \textbf{5} (2001), 147--163.

\bibitem{reineke}
\bysame, \emph{Quivers, desingularizations and canonical bases}, Studies in
  memory of {I}ssai {S}chur (Boston), Progress in Mathematics, vol. 210,
  Birkh\"auser, 2003, pp.~325--344.

\bibitem{open}
C.~Riedtmann and A.~Schofield, \emph{On open orbits and their complements}, J.
  Algebra \textbf{130} (1990), 388--411.

\bibitem{scheme}
C.~Riedtmann and G.~Zwara, \emph{Orbit closures and rank schemes}, Comment.
  Math. Helv. \textbf{88} (2013), 55--84.

\bibitem{quad}
C.~M. Ringel, \emph{Tame algebras and integral quadratic forms}, Lecture Notes
  in Mathematics, vol. 1099, Springer-Verlag, Berlin, 1984.

\bibitem{scho}
A.~Schofield, \emph{General representations of quivers}, Proc. London Math.
  Soc. \textbf{s3-65} (1992), no.~1, 46--64.

\bibitem{kavita2}
K.~Sutar, \emph{Resolutions of defining ideals of orbit closures for quivers of
  type {$\mathbb{A}_3$}}, J. Commut. Algebra \textbf{5} (2013), no.~3,
  441--475.

\bibitem{kavita}
\bysame, \emph{Orbit closures of source-sink {D}ynkin quivers}, Int. Math. Res.
  Not. \textbf{2015} (2015), no.~11, 3423--3444.

\bibitem{jerzy}
J.~Weyman, \emph{Cohomology of vector bundles and syzygies}, Cambridge Tracts
  in Mathematics, vol. 149, Cambridge University Press, Cambridge, 2003.

\bibitem{codim}
G.~Zwara, \emph{Orbit closures for representations of {D}ynkin quivers are
  regular in codimension two}, J. Math. Soc. Japan \textbf{57} (2005),
  859--880.

\bibitem{expo}
\bysame, \emph{Singularities of orbit closures in module varieties},
  Representations of Algebras and Related topics, EMS Series of Congress
  Reports, European Mathematical Society, 2011, pp.~661--725.

\end{thebibliography}
\bibliographystyle{amsplain}

\end{document}